\documentclass[preprint,12pt]{elsarticle}

\usepackage[utf8]{inputenc}
\usepackage{lineno,hyperref}
\usepackage{latexsym,enumerate}
\usepackage{amsmath,amsthm,amsopn,amstext,amscd,amsfonts,amssymb}
\usepackage{graphicx}
\usepackage{hyperref} 
\usepackage{xcolor}
\usepackage{geometry}
\usepackage{titlesec}
\usepackage{fancyhdr}
\usepackage[shortlabels]{enumitem}
\usepackage{mathrsfs}
\usepackage{tikz}
\usepackage{standalone}
\usepackage{multicol}
\usepackage{float}
\usepackage{svg}
\usepackage{verbatim}
\usepackage{amsmath,tikz}
\usetikzlibrary{matrix}
\usepackage{amsmath,graphicx}
\newtheorem{theorem}{Theorem}[section]

\newtheorem{remark}[theorem]{Remark}
\newtheorem{corollary}[theorem]{Corollary}
\newtheorem{example}[theorem]{Example}
\newtheorem{proposition}{Proposition}

\newtheorem*{observation*}{Observation}

\modulolinenumbers[5]

\journal{TBA}

\begin{document}

\begin{frontmatter}

\title{Graph Theoretic Framework of Dynamical Systems Through the Boundary Polynomial}

\author[mymainaddress]{Marcos Masip}
\ead{marcosmasipresearch@gmail.com}

\address[mymainaddress]{Department of Mathematics and Statistics, Florida International University,
Miami, Florida 33199, USA}

\begin{abstract}
This work is devoted to the study of the relationships between graph theory and the qualitative analysis of ordinary differential equations, with a special focus on two-dimensional systems. In particular, we reinterpret classical results through the lens of boundary polynomials of graphs. The theory naturally leads to questions about limit cycles, which arise in many processes in nature and remain a central object of study in dynamical systems. The significance of limit cycles is underscored by their place in Hilbert’s 16th problem, one of the unsolved challenges from his famous list of 23 problems posed in 1900.
\end{abstract}

\begin{keyword}
\texttt{dynamical systems; boundary polynomial; weighted boundary polynomial; limit cycles; hopf bifurcation; liapunov number.}%\sep \LaTeX\sep Elsevier \sep template
\MSC[2020] 37C20\sep 05C31\sep 37G15
\end{keyword}

\end{frontmatter}

\linenumbers

\section{Introduction.}
This paper explores the mathematical intersection between graph theory and the qualitative analysis of dynamical systems governed by ordinary differential equations (ODEs). The focus is particularly on two-dimensional systems, where graphical representations and bifurcation theory provide deep insights into the behavior and stability of solutions. The rich structure of these systems allows for the application of classical results, including the Hartman–Grobman Theorem, the Stable Manifold Theorem, and Hopf bifurcation theory, among others. 

Within this framework, we use the \emph{boundary polynomials of graphs} as defined in \cite{Carballosa2025}, and demonstrate how they serve as a unifying tool for studying dynamical behavior. Beyond detecting phenomena such as limit cycles, structural stability, and bifurcation values, the boundary polynomial also provides a new perspective for re-examining classical two-dimensional systems. By recasting familiar results through this graph-theoretic lens, we aim to reveal structural patterns that are often obscured in the traditional analytic setting.

This approach lays the foundation for future investigations into computational methods that leverage graph-theoretic techniques to analyze nonlinear dynamics and approximate solutions to open problems in the theory of differential equations. It also provides a potential tool for determining the number of limit cycles, for example, the maximum number of limit cycles in quadratic systems (see \cite{Bautin} for local results).

Before delving into the general theory, let us examine a few concrete examples. They demonstrate how even small graphs can reveal surprisingly rich behavior when viewed through the lens of boundary polynomials. Later we introduce a weighted extension of the boundary polynomial, where each subset of vertices is assigned a weight. This weighted boundary polynomial will allow us to incorporate cancellations and signed contributions, enriching the algebraic structure and enabling connections to more subtle dynamical phenomena.

\begin{flushleft}
Let $G$ be a graph of order $n$. The \emph{boundary polynomial} of $G$ with variable $x$ and $y$ is defined in \cite{Carballosa2025} as follows:
\end{flushleft}
\begin{equation}\label{Poly1}
    B(G;x,y) = \sum_{S \subseteq V(G)}x^{|B(S)|}y^{|S|},
\end{equation}

\begin{flushleft}
or equivalently,
\end{flushleft}

\begin{equation}\label{Poly2}
    B(G;x,y) = \sum_{0 \leq i+j \leq n}B_{i,j}(G)x^{i}y^{j},
\end{equation}

\begin{flushleft}
where $B_{i,j}(G)$ represents the number of subsets of the vertex in $G$ with $|B(S)| = i$ and $|S| = j$.
\end{flushleft}

\begin{flushleft}
For the complete graph $K_{n}$ and the empty graph $E_{n}$ of order $n$, respectively, we have the following expressions:
\end{flushleft}

\begin{enumerate}[i)]\label{eq:Poly1}
    \item $B(K_{n};x,y) = (x+y)^{n}+1-x^{n}$.
    \item $B(E_{n};x,y) = (1+y)^{n}$.
 \end{enumerate}

\begin{flushleft}
For the rest of the work we are considering $x = x(t)$ and $y = y(t)$, $t \in \mathbb{R}$.
\end{flushleft}

\begin{flushleft}
Let $G_{1}$ be a graph of order $n_{1}$ and $G_{2}$ be a graph of order $n_{2}$. Consider the system:
\begin{equation}\label{eq:example1.1}
\frac{dx}{dt} = 2^{n_{1}}x - B(G_{1};x,y), \frac{dy}{dt} = 2^{n_{2}}y - B(G_{2};x,y),
\end{equation}
\end{flushleft}
\begin{example}

Considering $G_{1}$ to be the family of complete graphs, and $G_{2}$ the family of empty graphs, then we have the following results:

\begin{enumerate}[label=\alph*), font=\bfseries]
    \item If $G_{2}$ has at least three vertices and $G_{1}$ has either exactly one vertex or two vertices, then the system has a topological saddle at $(1,1)$. This is also the case for when $G_{1}$ has at least three vertices and $G_{2}$ has exactly one vertex.
    \item If $G_{2}$ has exactly one vertex and $G_{1}$ has either exactly one vertex or two vertices, then $(1,1)$ is an unstable node for the system.
    \item If both $G_{1}$ and $G_{2}$ have at least three vertices, respectively, then $(1,1)$ is a stable node for the system.
    \item If $G_{2}$ has exactly two vertices, then the system is degenerate at $(1,1)$.
\end{enumerate}
\end{example}

\begin{proof}
    If $G$ is a graph of order $n$, then $B(G;1,1) = 2^{n}$. Thus, $2^{n_{1}} - B(G_{1};1,1) = 0$ and $2^{n_{2}} - B(G_{2};1,1) = 0$, which shows that $(1,1)$ is an equilibrium point, and is of interest. The Jacobian matrix $Df(x,y)$ for the system takes the form:

    $$Df(x,y) = \begin{bmatrix}
        2^{n_{1}}-n_{1}(x+y)^{n_{1}-1}+n_{1}x^{n_{1}-1} & -n_{1}(x+y)^{n_{1}-1}\\
        0 & 2^{n_{2}}-n_{2}(1+y)^{n_{2}-1}
    \end{bmatrix}.$$
We note that if $n_{1} = n_{2} = 1$, then the upper triangular matrix becomes $$\begin{bmatrix}
    2 & -1\\0 & 1
\end{bmatrix},$$ and the eigenvalues are strictly positive, which implies that for all $(x,y) \in \mathbb{R}^{2}$, the system is unstable. If $n_{1} = 1$, $n_{2} > 1$ (same as $n_{2} \geq 2$), then $$Df(1,1) = \begin{bmatrix}
    2 & -1\\ 0 & 2^{n_{2}}-2^{n_{2}-1}n_{2}
\end{bmatrix},$$ and thus $\lambda_{1} = 2$, $\lambda_{2} = 2^{n_{2}}-2^{n_{2}-1}n_{2}$ are the eigenvalues. For $(1,1)$ to be a saddle we need $\lambda_{2} < 0$, which implies $2 < n_{2}$. For $(1,1)$ to be an unstable node, we need $\lambda_{2} > 0$ which implies $2 > n_{2}$, which is impossible in this case. The $(1,1)$ becomes degenerate at $n_{2} = 2$. If $n_{1} > 1$, $n_{2} = 1$, then $$Df(x,y) = \begin{bmatrix}
    2^{n_{1}}-2^{n_{1}-2}n_{1} + n_{1} & -2^{n_{1}-1}n_{1}\\ 0 & 1
\end{bmatrix},$$ with eigenvalues $\lambda_{1} = 2^{n_{1}}-2^{n_{1}-2}n_{1} + n_{1}$, $\lambda_{2} = 1$. For $(1,1)$ to be a saddle we need $\lambda_{1} < 0$, which simplifies to $n_{1} \geq 3$. For $(1,1)$ to be an unstable node, we need $\lambda_{1} > 0$ which only happens for $n_{1} = 1$ or $n_{1} = 2$. Degeneracy in this case is not possible since $n_{1}$ is a positive integer. For $n_{1} > 1$, $n_{2} > 1$, then $$Df(x,y) = \begin{bmatrix}
    2^{n_{1}}-2^{n_{1}-1}n_{1}+n_{1} & -2^{n_{1}-1}n_{1}\\ 0 & 2^{n_{2}} - 2^{n_{2}-1}n_{2}
\end{bmatrix}$$ with eigenvalues $\lambda_{1} = 2^{n_{1}}-2^{n_{1}-1}n_{1}+n_{1}$, $\lambda_{2} = 2^{n_{2}} - 2^{n_{2}-1}n_{2}$. For $(1,1)$ to be a saddle, we can consider $\lambda_{1} > 0$, $\lambda_{2} < 0$ which simplifies to $n_{1} = 2$, $n_{2} > 2$. We can also consider the case $\lambda_{1} < 0$, $\lambda_{2} > 0$, but simplifies to $n_{1} \geq 3$, $n_{2} < 2$ (same as $n_{2} = 1$) which is not possible. For $(1,1)$ to be an unstable node we need $n_{2} < 2$ which is not possible. On the other hand, for $(1,1)$ to be a stable node we need $\lambda_{1} < 0$, $\lambda_{2} < 0$ which simplifies to $n_{1} \geq 3$, $n_{2} > 2$ (same as $n_{2} \geq 3$). For degeneracy, having $n_{2} = 2$ is enough, since there does not exist $n_{1}$ such that $\lambda_{1} = 0$.

\end{proof}

\begin{figure}[H]
  \centering

  % ===================== Row 1 =====================
  \resizebox{0.65\textwidth}{!}{%
    \begin{minipage}[t]{0.32\textwidth}
      \centering
      \begin{tikzpicture}[every node/.style={circle, draw, semithick, minimum size=4mm, inner sep=1pt, font=\scriptsize}]
        \node (a1) at (-0.9, 0) {$v_1$};
        \node (a2) at ( 0.9, 0) {$v_2$};
        \draw[semithick] (a1) -- (a2);
      \end{tikzpicture}
      \vspace{0.25em}\par\footnotesize $G_{1}=K_{2}$
    \end{minipage}\hfill
    \begin{minipage}[t]{0.32\textwidth}
      \centering
      \begin{tikzpicture}[every node/.style={circle, draw, semithick, minimum size=4mm, inner sep=1pt, font=\scriptsize}]
        \node (b1) at ( 0,  1.0) {$u_1$};
        \node (b2) at (-1.1,-0.7) {$u_2$};
        \node (b3) at ( 1.1,-0.7) {$u_3$};
      \end{tikzpicture}
      \vspace{0.25em}\par\footnotesize $G_{2}=E_{3}$
    \end{minipage}\hfill
    \begin{minipage}[t]{0.32\textwidth}
      \centering
      \includegraphics[width=\linewidth]{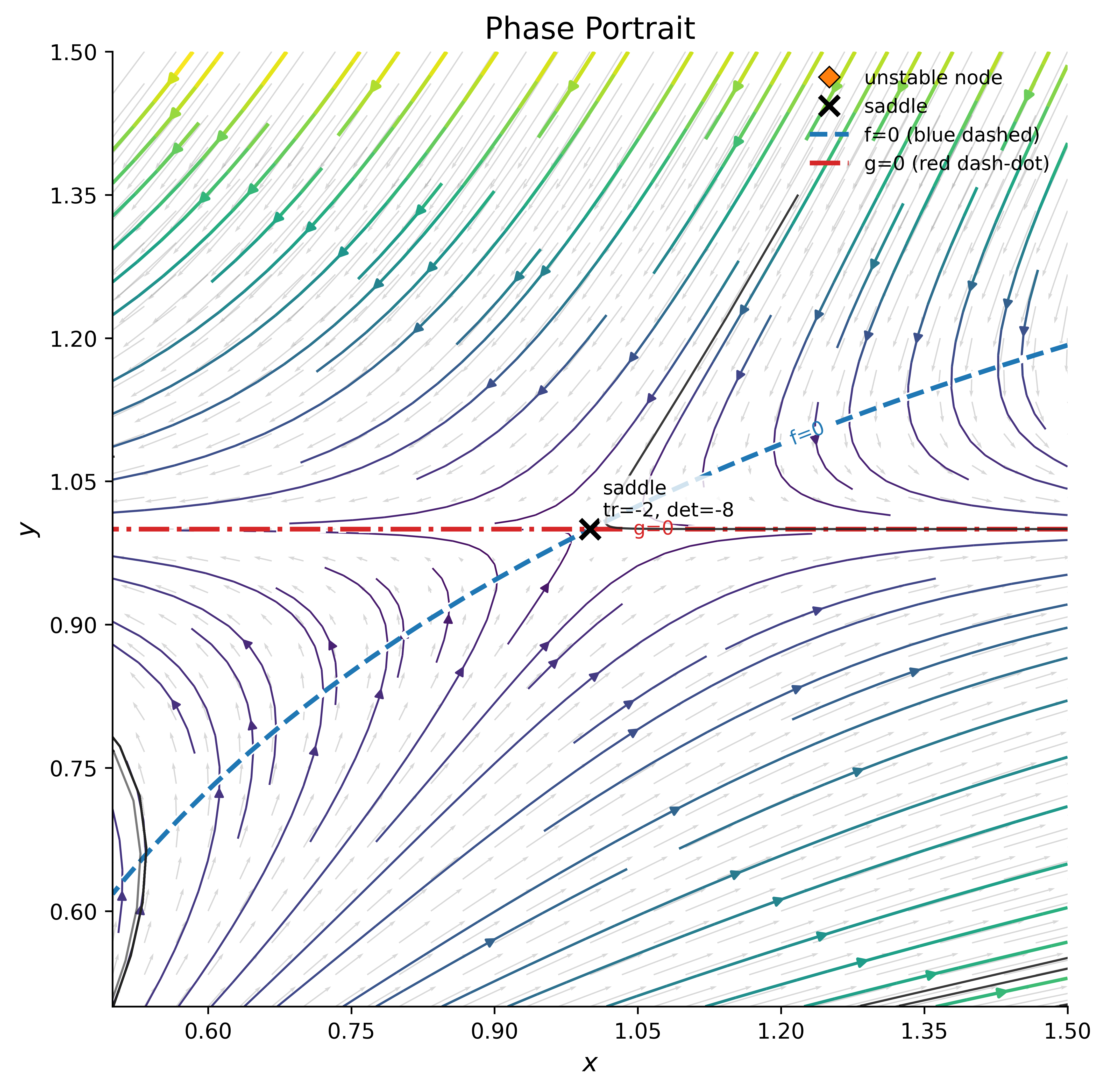}
      \par\footnotesize Topological saddle
    \end{minipage}%
  }

  \vspace{0.8em}

  % ===================== Row 2 =====================
  \resizebox{0.65\textwidth}{!}{%
    \begin{minipage}[t]{0.32\textwidth}
      \centering
      \begin{tikzpicture}[every node/.style={circle, draw, semithick, minimum size=4mm, inner sep=1pt, font=\scriptsize}]
        \node (c1) at (0,0) {$v_1$};
      \end{tikzpicture}
      \vspace{0.25em}\par\footnotesize $G_{1}=K_{1}$
    \end{minipage}\hfill
    \begin{minipage}[t]{0.32\textwidth}
      \centering
      \begin{tikzpicture}[every node/.style={circle, draw, semithick, minimum size=4mm, inner sep=1pt, font=\scriptsize}]
        \node (d1) at (0,0) {$u_1$};
      \end{tikzpicture}
      \vspace{0.25em}\par\footnotesize $G_{2}=E_{1}$
    \end{minipage}\hfill
    \begin{minipage}[t]{0.32\textwidth}
      \centering
      \includegraphics[width=\linewidth]{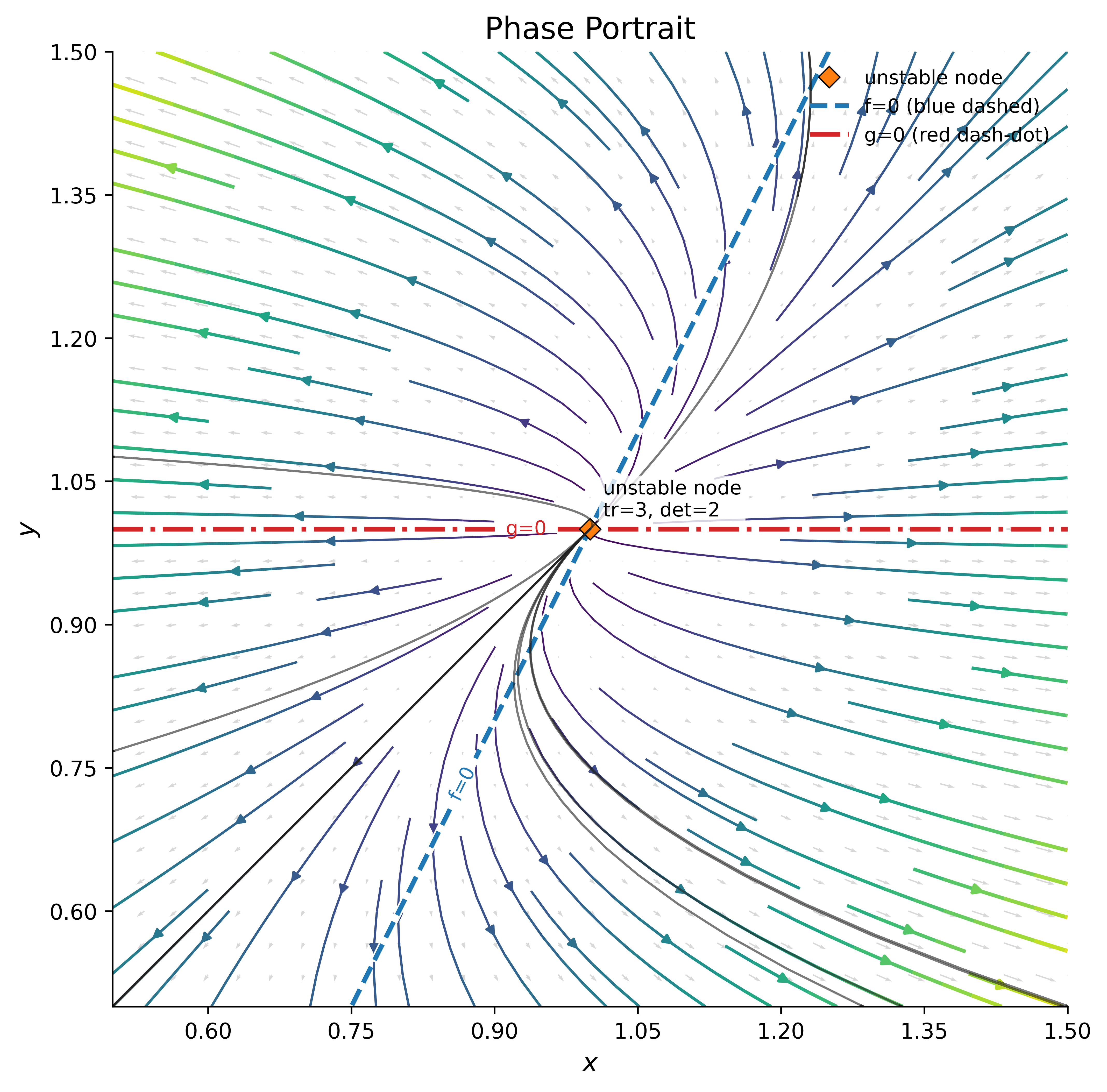}
      \par\footnotesize Unstable node
    \end{minipage}%
  }

  \vspace{0.8em}

  % ===================== Row 3 =====================
  \resizebox{0.65\textwidth}{!}{%
    \begin{minipage}[t]{0.32\textwidth}
      \centering
      \begin{tikzpicture}[every node/.style={circle, draw, semithick, minimum size=4mm, inner sep=1pt, font=\scriptsize}]
        \node (e1) at (0, 1.2) {$v_1$};
        \node (e2) at (-1.04, -0.6) {$v_2$};
        \node (e3) at ( 1.04, -0.6) {$v_3$};
        \draw[semithick] (e1)--(e2);
        \draw[semithick] (e2)--(e3);
        \draw[semithick] (e3)--(e1);
      \end{tikzpicture}
      \vspace{0.25em}\par\footnotesize $G_{1}=K_{3}$
    \end{minipage}\hfill
    \begin{minipage}[t]{0.32\textwidth}
      \centering
      \begin{tikzpicture}[every node/.style={circle, draw, semithick, minimum size=4mm, inner sep=1pt, font=\scriptsize}]
        \node (f1) at (-0.9,  0.9) {$u_1$};
        \node (f2) at ( 0.9,  0.9) {$u_2$};
        \node (f3) at (-0.9, -0.9) {$u_3$};
        \node (f4) at ( 0.9, -0.9) {$u_4$};
      \end{tikzpicture}
      \vspace{0.25em}\par\footnotesize $G_{2}=E_{4}$
    \end{minipage}\hfill
    \begin{minipage}[t]{0.32\textwidth}
      \centering
      \includegraphics[width=\linewidth]{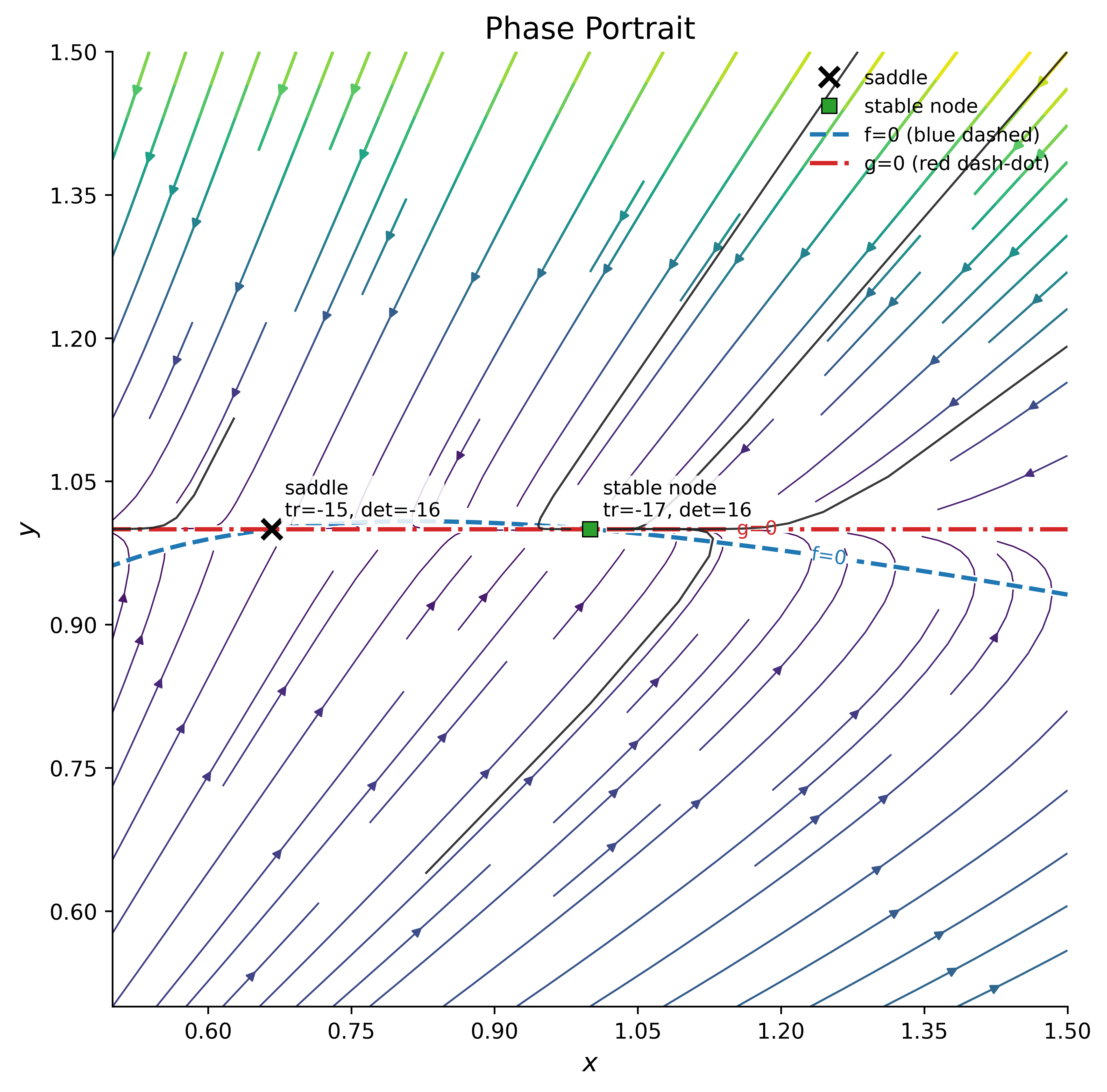}
      \par\footnotesize Stable node
    \end{minipage}%
  }

  % ===================== Caption =====================
  \caption{Examples of $G_{1}$ vs.\ $G_{2}$ for (\ref{eq:example1.1}) along with their respective stability for $(1,1)$.}
  \label{fig:three_layers_graphs_images}
\end{figure}
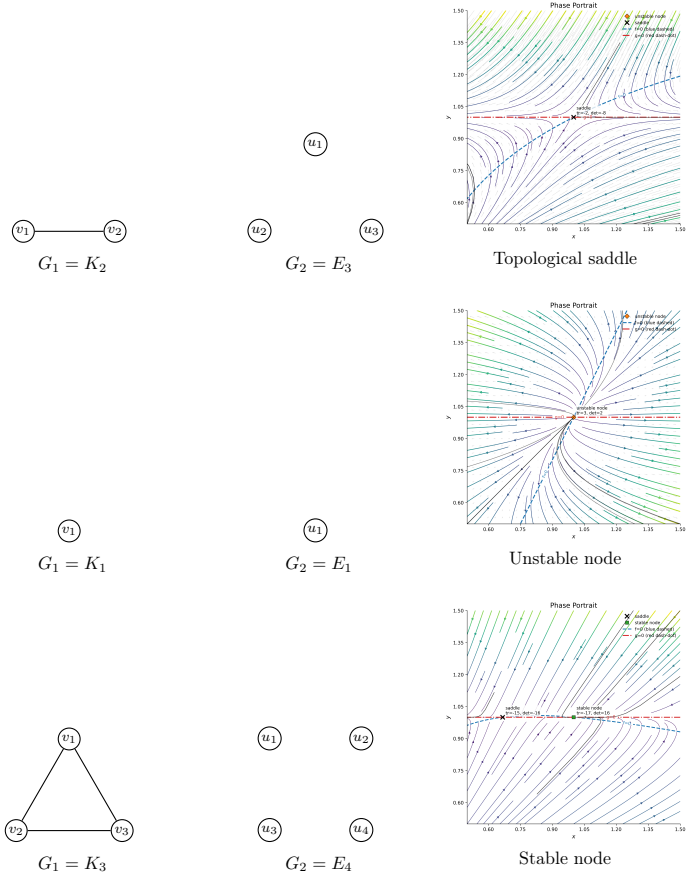

Now, let $G$ be a graph of order $n$ and let $w: 2^{V(G)} \to \{1, k\} \subset \mathbb{R}$ be a weight function on subsets of vertices defined by,

\begin{equation}\label{weight:function}
    w(S) = \begin{cases}
    1, & \text{ if } S = \emptyset\text{ or } |S| \text{ is odd} \\[6pt]
    k, & \text{ otherwise }

    \end{cases},
\end{equation}

\begin{flushleft}
and define the weighted boundary polynomial of a graph $G$ of order $n$ to be
\end{flushleft}

\begin{equation}\label{weight:2}
    B_{w}(G;x,y) = \sum_{0 \leq i+j \leq n}B_{i,j}^{(w)}(G)x^{i}y^{j},
\end{equation}
where
\begin{equation}\label{weight:sum_of_weights}
    B_{i,j}^{(w)}(G) = \begin{cases}
        \sum_{S \subseteq V(G),|B(S)|=i,|S|=j}w(S), \text{ if there exists an } S \text{ such that } |B(S)| = i \text{ and } |S| = j,\\[6pt]
        \sum_{\ell=1}^{\frac{n-3}{2}}k_{\ell}, \text{ if } |B(S)| = i, |S| = j \text{ does not exist and } 3 < n \text{ is odd}, \\[6pt]
        k^{\ast}, \text{ otherwise }
    \end{cases}
\end{equation}
with $\{k^{\ast}\} \cup \{k_{\ell}\}_{\ell=1}^{\frac{n-3}{2}} \subset \mathbb{R}$.
\newline

\begin{observation*}
    Just like for the boundary polynomial, when expanding the weighted boundary polynomial, we don't consider $B_{\ell,0}^{(w)}(G)$ for $\ell \not= 0$. That is, you can consider $B_{\ell,0}^{(w)}(G) = 0$ for $\ell \not= 0$.
\end{observation*}

We should take note on the following notations:

\begin{equation}
\begin{aligned}
    B(G;x) &= B(G;x,x),\\
    B_{w}(G;x) &= B_{w}(G;x,x),\\
    \widetilde{B}(G;x,y) &= B(G;x,y) - 1, \\
    \widetilde{B}_{w}(G;x,y) &= B_{w}(G;x,y) - 1.
\end{aligned}
\end{equation}

For the second example, we will use this graph theoric perspective on the classical Lienard system of the form

\begin{equation}\label{Lienard:1}
    \frac{dx}{dt} = y - \widetilde{B}_{w}(G_{1};x), \frac{dy}{dt} = -\widetilde{B}_{w}(G_{2};x),
\end{equation}
 where $G_{1}$, $G_{2}$ are graphs of finite order.

\begin{theorem}\label{thm:1.2}
The following results are concerning the origin of (\ref{Lienard:1}):

    \begin{enumerate}[label=\textbf{\alph*)}]
    \item The system (\ref{Lienard:1}) with $G_{2} = E_{1}$, $G_{1}$ a graph of order $3$, and $B_{0,1}^{(w)}(G_{1})(B_{1,2}^{(w)}(G_{1}) + B_{2,1}^{(w)}(G_{1}) + B_{0,3}^{(w)}(G_{1})) < 0$ has exactly one limit cycle which is stable if $B_{0,1}^{(w)}(G_{1}) < 0$ and unstable if $B_{0,1}^{(w)}(G_{1}) > 0$.

    \item Considering $G_{2} = E_{1}$ and $G_{1}$ a graph of odd order $m$, then the system (\ref{Lienard:1}) has at most $\frac{m-1}{2}$ local limit cycles.

    \item For $\epsilon \not= 0$ sufficiently small, the system (\ref{Lienard:1}) with $G_{2} = E_{1}$, and placing $\epsilon\widetilde{B}(G_{1};x)$ in the system instead of $\widetilde{B}(G_{1};x)$ with $G_{1}$ of odd order $m$, the system has at most $\frac{m-1}{2}$ limit cycles.

    \item Following item $c)$, with $G_{2} = E_{1}$ and $G_{1}$ of odd order $m$, for $\epsilon \not= 0$ sufficiently small, the system has exactly $\frac{m-1}{2}$ limit cycles which are asymptotic to circles of radius $r_{j}$, $j=1,\ldots,\frac{m-1}{2}$, centered at the origin as $\epsilon \to 0$ iff the $(\frac{m-1}{2})th$ degree polynomial

    \begin{equation}\label{poly:perko}
        \frac{1}{2}B_{0,1}^{(w)}(G_{1}) + \cdots + \binom{m+1}{\frac{m+1}{2}}\frac{\sum_{i+j=m}B_{i,j}^{(w)}(G_{1})}{2^{m+1}}\rho^{\frac{m-1}{2}} = 0
    \end{equation}

    has $\frac{m-1}{2}$ positive roots $\rho = r_{j}^{2}$, $j = 1,\ldots,\frac{m-1}{2}$.
    
\end{enumerate}
\end{theorem}

\begin{proof}
    These are reformulated results, where a) comes from \cite{LinsDeMeloPugh1977}, items c) and d) from \cite{Perko2001}, and item b) from \cite{BlowsLloyd1984} and \cite{LloydLynch1988}.
\end{proof}

From now on the system that item c) and item d) are utilizing in Theorem (\ref{thm:1.2}), will be referred to as the $\epsilon$-system of (\ref{Lienard:1}).

Suppose we want to find a system in the form of the $\epsilon$-system of (\ref{Lienard:1}) with exactly two limit cycles asymptotic to circles of radius $r = \sqrt{\frac{1}{2}}$, $r = \sqrt{3}$. Thus, we simply set $(\rho - \frac{1}{2})(\rho - 3) = \rho^{2} -\frac{7}{2}\rho + \frac{3}{2}$ equal to (\ref{poly:perko}). That is,

$$
\rho^{2} -\frac{7}{2}\rho + \frac{3}{2} = \frac{5}{16}(\sum_{i+j = 5}B_{i,j}^{(w)}(G_{1}))\rho^{2} + \frac{3}{8}(\sum_{i+j = 3}B_{i,j}^{(w)}(G_{1}))\rho + \frac{1}{2}(\sum_{i+j = 1}B_{i,j}^{(w)}(G_{1}))
$$

Hence, we need $\sum_{i+j = 5}B_{i,j}^{(w)}(G_{1}) = \frac{16}{5}$, $\sum_{i+j = 3}B_{i,j}^{(w)}(G_{1}) = -\frac{28}{3}$, $\sum_{i+j = 1}B_{i,j}^{(w)}(G_{1}) = 3$. Which can be satisfied by considering $G_{1}$ to be the graph shown on the right-hand side of Figure 2 with $k = -\frac{211}{80}$, $k_{1} = \frac{299}{120}$.

\begin{figure}[H]
  \centering
  \begin{minipage}[t]{6cm}
    \includegraphics[width=0.7\linewidth]{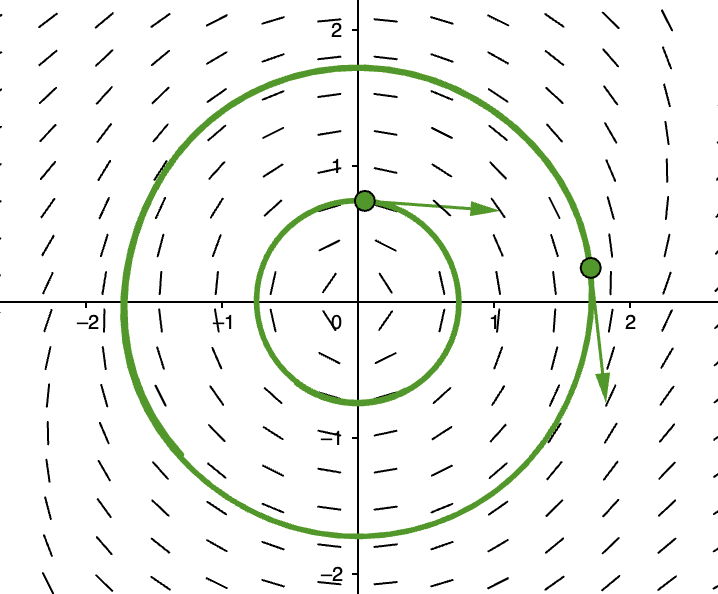}  
  \end{minipage}
%\hfill
  \begin{minipage}[t]{6cm}
   \centering
     \begin{tikzpicture}[scale=0.7, every node/.style={
      circle, draw, semithick,
      minimum size=5.0mm,
      inner sep=1pt,
      font=\small
    }
  ]
    % Vertices
    \node (v1) at (0,  2) {$v_1$};
    \node (v2) at (-2, 1) {$v_2$};
    \node (v3) at (-2,-1) {$v_3$};
    \node (v4) at (0, -2) {$v_4$};
    \node (v5) at (2, -1) {$v_5$};

    % Edge(s)
    \draw[semithick] (v1) -- (v5);
  \end{tikzpicture} 
  \end{minipage}
    \caption{Two limit cycles (left) with radius $r = \sqrt{\frac{1}{2}}$ and $r = \sqrt{3}$, respectively, and the associated graph (right).}
    \label{fig:placeholder}
\end{figure}
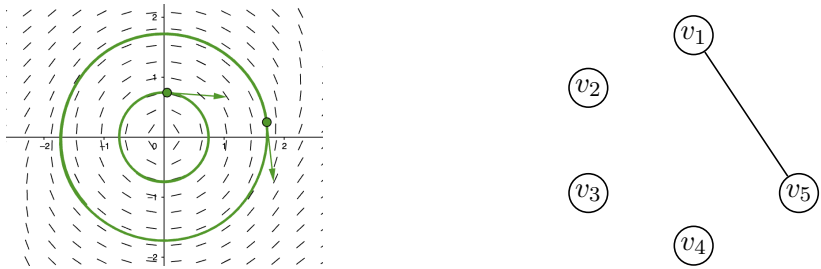

\newpage

\begin{flushleft}
One of the clear consequences of Theorem \ref{thm:1.2} is the following:
\end{flushleft}

\begin{corollary}\label{cor:1.3}
    If $G_{1}$ is of order $3$ and has at least one isolated vertex, then the associated limit cycle in (\ref{Lienard:1}) will always be unstable.
\end{corollary}

\begin{proof}
    If $G_{1}$ of order three has at least one isolated vertex, then $B_{0,1}^{(w)}(G_{1}) > 0$.
\end{proof}

\begin{figure}[H]
  \centering
  % ===================== Row 1 =====================
  \resizebox{0.85\textwidth}{!}{%
    \begin{minipage}[t]{0.47\textwidth}
      \centering
      \begin{tikzpicture}[scale=0.95, every node/.style={circle, draw, semithick, minimum size=5mm, inner sep=1pt, font=\small}]
        \node (v1) at (-0.9, 0) {$v_1$};
        \node (v2) at ( 0.9, 0) {$v_2$};
        \node (v3) at ( 0.0,-1.2) {$v_3$}; % isolated
        \draw[semithick] (v1) -- (v2);
      \end{tikzpicture}
      \vspace{0.3em}\par\small $G_{1}$
    \end{minipage}\hfill
    \begin{minipage}[t]{0.47\textwidth}
      \centering
      \includegraphics[width=0.7\linewidth]{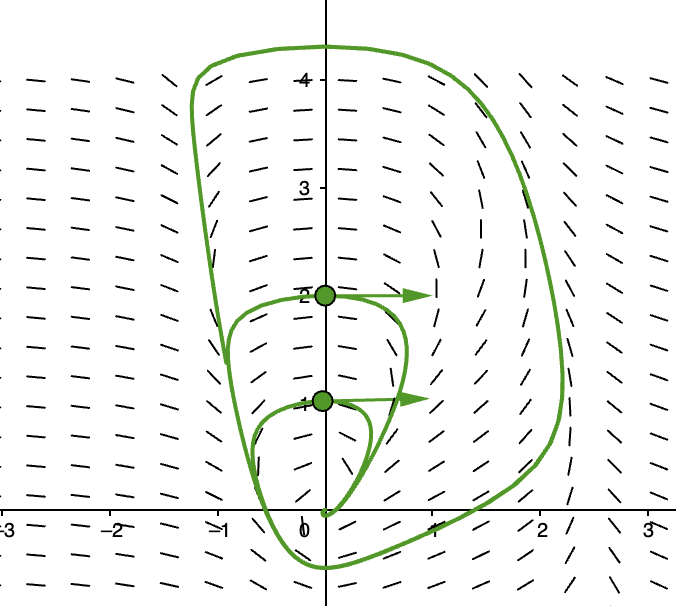}
      \vspace{0.3em}\par\small Unstable limit cycle
    \end{minipage}%
  }

  \vspace{1em} % space between rows

  % ===================== Row 2 =====================
  \resizebox{0.85\textwidth}{!}{%
    \begin{minipage}[t]{0.47\textwidth}
      \centering
      \begin{tikzpicture}[scale=0.95, every node/.style={circle, draw, semithick, minimum size=5mm, inner sep=1pt, font=\small}]
        \node (w1) at (0, 1.2) {$v_1$};
        \node (w2) at (-1.04, -0.6) {$v_2$};
        \node (w3) at ( 1.04, -0.6) {$v_3$};
        \draw[semithick] (w1) -- (w2);
        \draw[semithick] (w2) -- (w3);
        \draw[semithick] (w3) -- (w1);
      \end{tikzpicture}
      \vspace{0.3em}\par\small $G_{1}=K_{3}$
    \end{minipage}\hfill
    \begin{minipage}[t]{0.47\textwidth}
      \centering
      \includegraphics[width=0.7\linewidth]{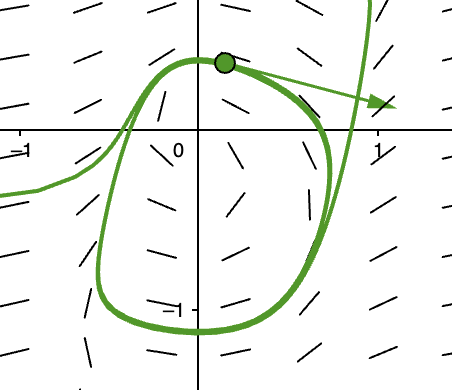}
      \vspace{0.3em}\par\small Stable limit cycle
    \end{minipage}%
  }

  % ===================== Caption =====================
  \caption{Examples of $G_{1}$ (left) associated with the stable/unstable limit cycle (right);
  \textbf{First row}: $k = 0$, $k^{\ast} = -2$;
  \textbf{Second row}: $k = 0$, $k^{\ast} = -1$, for the system~(\ref{Lienard:1}).}
  \label{fig:G1_edge_iso_image}
\end{figure}
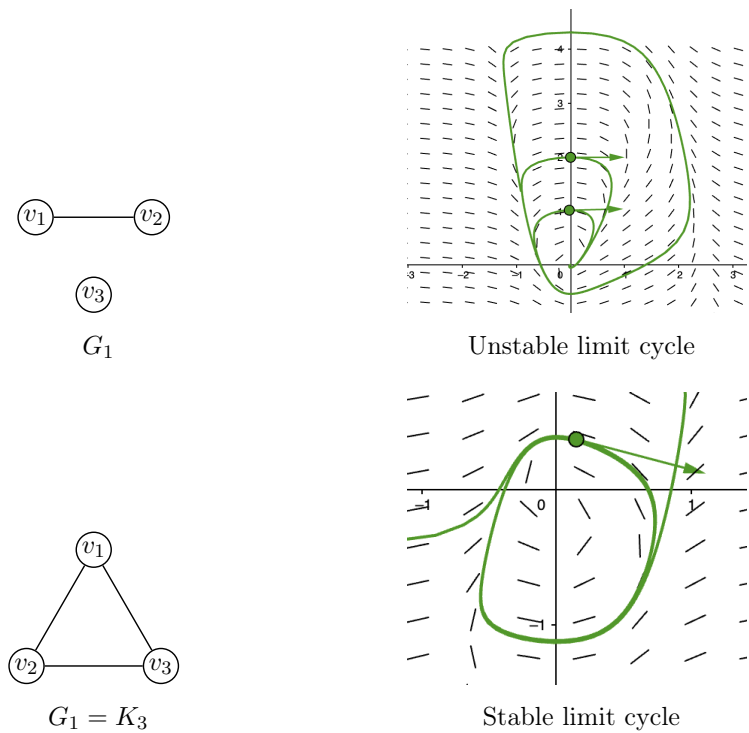

\section{Graphs not associated to limit cycles for the Liénard system}\label{sect2}

\begin{flushleft}
Since (\ref{thm:1.2}) shows how to associate graphs to limit cycles, it can also tell us when certain graphs cannot. For example, as a consequence we have the following corollary:
\end{flushleft}

\begin{corollary}\label{cor:1.4}
    The complete graph $K_{5}$ is not associated to any two limit cycles in the $\epsilon$-system of (\ref{Lienard:1}) asymptotic to circles of radius $r_{1},r_{2}$.
\end{corollary}

\begin{proof}
    Let the two limit cycles be asymptotic to circles of radius $r_{1} = m_{1}$ and $r_{2} = m_{2}$, with $m_{1},m_{2} > 0$. Then $B_{0,1}^{(w)}(G_{1}) = 2m_{1}^{2}m_{2}^{2}$, $\sum_{i+j=3}B_{i,j}^{(w)}(G_{1}) = -\frac{8}{3}(m_{1}^{2}+m_{2}^{2})$, $\sum_{i+j=5}B_{i,j}^{(w)}(G_{1}) = \frac{16}{5}$. Since $G_{1} = K_{5}$, this implies that $k = -\frac{64}{75}$ and $2m_{1}^{2}m_{2}^{2} = -8(m_{1}^{2}+m_{2}^{2})$ which is impossible.
\end{proof}

\begin{flushleft}
    We can generalize this result in the following interesting theorem:
\end{flushleft}

\begin{theorem}\label{thm:1.5}
    The complete graph $K_{m}$ and the empty graph $E_{m}$ with $3 < m$ odd, are not associated to any $\frac{m-1}{2}$ limit cycles in the $\epsilon$-system of (\ref{Lienard:1}) asymptotic to circles of radius $r_{j}$, $j = 1,\ldots,\frac{m-1}{2}$.
\end{theorem}

\begin{proof}
    Consider the case $G_{1} = K_{m}$ with $3 < m$ odd. Suppose we want to find the exact $\frac{m-1}{2}$ corresponding limit cycles asymptotic to circles of radius $r_{1} = a_{1}$, $r_{2} = a_{2}, \ldots$, $r_{\frac{m-1}{2}} = a_{\frac{m-1}{2}}$ where $a_{1},a_{2},\ldots,a_{\frac{m-1}{2}} > 0$. Then we consider,
    
    $$\prod_{i=1}^{\frac{m-1}{2}}(\rho-a_{i}^{2}) = \rho^{\frac{m-1}{2}} - (\sum_{i}a_{i}^{2})\rho^{\frac{m-3}{2}} + (\sum_{i < j}a_{i}^{2}a_{j}^{2})\rho^{\frac{m-5}{2}} - (\sum_{i<j<k}a_{i}^{2}a_{j}^{2}a_{k}^{2})\rho^{\frac{m-7}{2}} + \cdots + (-1)^{\frac{m-1}{2}}\prod_{i}a_{i}^{2},$$

to be equal to the following,

    $$\frac{1}{2}B_{0,1}^{(w)}(G_{1}) + \cdots + A_{m-1}\frac{\sum_{i+j=m-2}B_{i,j}^{(w)}(G_{1})}{2^{m-1}}\rho^{\frac{m-3}{2}} + A_{m+1}\frac{\sum_{i+j=m}B_{i,j}^{(w)}(G_{1})}{2^{m+1}}\rho^{\frac{m-1}{2}},$$

where $A_{d} = \binom{d}{\frac{d}{2}}$. Then we have, 

$$\sum_{i+j = m}B_{i,j}^{(w)}(G_{1}) = 2^{m+1} A_{m+1}^{-1},$$ 

$$\sum_{i+j = m-2}B_{i,j}^{(w)}(G_{1}) = -2^{m-1} A_{m-1}^{-1}(\sum_{i}a_{i}^{2}),$$ 

$$\sum_{i+j = m-4}B_{i,j}^{(w)}(G_{1}) = 2^{m-3}A_{m-3}^{-1}(\sum_{i < j}a_{i}^{2}a_{j}^{2}),$$

$$\vdots$$

$$B_{0,1}^{(w)}(G_{1}) = 2(-1)^{\frac{m-1}{2}}\prod_{i}a_{i}^{2},$$

where it yields,

$$k = \left(2^{m+1}A_{m+1}^{-1} - \sum_{j=0}^{\frac{m-1}{2}}\binom{m}{m-2j}\right)\left(\sum_{j=0}^{\frac{m-3}{2}}\binom{m}{m-(2j+1)}\right)^{-1},$$

and the following is impossible,

$$-2^{m-1}(m-2)^{-1}A_{m-1}^{-1}\left(\sum_{i}a_{i}^{2}\right) = 2^{m-3}(m-4)^{-1}A_{m-3}^{-1}\left(\sum_{i < j}a_{i}^{2}a_{j}^{2}\right).$$

Now, considering the case $G_{1} = E_{m}$ and following the same structure above also gives us,

$$\frac{2^{m+1}A_{m+1}^{-1}-1}{m-1} = \frac{-2^{m-1}A_{m-1}^{-1}(\sum_{i}a_{i}^{2})-\frac{m(m-1)}{2}}{m-3},$$

which is impossible since positive $\not=$ negative.
\end{proof}

\section{Stability and Hopf Bifurcation}\label{sect3}

\begin{flushleft}
For most of the results we have been talking about odd order graphs and the relationship they have with limit cycles. But we have not mention anything concerning even order graphs and their possible relationship with limit cycles. This actually depends on the type of system we construct using either the boundary polynomial or the weighted boundary polynomial. For the rest of the work, everything is mostly developed with the boundary polynomials, but we use the weighted boundary polynomial when necessary. Necessecity in this case means negative coefficients that (\ref{Poly1}) cannot give. Next, I want to present a system that gives us the possibility to associate complete graphs of even order to a stable limit cycle. Consider the system,
\end{flushleft}

\begin{equation}\label{system:1}
    \frac{dx}{dt} = -B(G;x,y)y + x(1 - x^{2} - y^{2}), \frac{dy}{dt} = B(G;x,y)x + y(1 - x^2 - y^2),
\end{equation}

\begin{flushleft}
where $G$ is a graph of order $n$. This system (\ref{system:1}) is not new. In fact, many similar systems have been used through the years to emphasize on the stability and instability of periodic solutions. But (\ref{system:1}) allows us to say the following:
\end{flushleft}

\begin{proposition}
If $G$ is the family of complete graphs with even order, then the unit circle is a stable limit cycle for (\ref{system:1}).
\end{proposition}

\begin{proof}
Note that $B(G = K_{n};x,y)$ can't be zero on the unit circle and thus if $n$ would have been odd, then we have $B(G = K_{n};0,-1) = -1 + 1 = 0$. Thus, we continue by considering $n$ to be even. We have that $\gamma(t) = (\cos(t),\sin(t))^{T}$ is a periodic solution to $(\ref{system:1})$. The derivative of the Poincare map along a straight line normal to $\{\mathbf{x} \in \mathbf{R}^{2}: \mathbf{x} = \gamma(t) - \gamma(0), 0 \leq t \leq 2\pi\}$ at $\mathbf{x} = 0$ is given by $e^{\int_{0}^{2\pi}(\frac{\partial}{\partial x}(\frac{dx}{dt}) + \frac{\partial}{\partial y}(\frac{dy}{dt}))(\gamma(t))dt}$. But since $\int_{0}^{2\pi}(\frac{\partial}{\partial x}(\frac{dx}{dt}) + \frac{\partial}{\partial y}(\frac{dy}{dt}))(\gamma(t))dt = -4\pi < 0$, then the periodic solution is a stable limit cycle.
\end{proof}

\begin{flushleft}
In fact, the unit circle will always be a periodic solution to (\ref{system:1}) as long as $B(G;x,y) \not= 0$ on the unit circle, for a chosen graph $G$. We continue by constructing a boundary polynomial based system such that the limit cycle is stable independent of the graph chosen. That is, we can consider:
\end{flushleft}

\begin{equation}\label{system:2}
\frac{dx}{dt} = -y + B(G;x,y)(x - x^{3} - x y^{2}), \frac{dy}{dt} = x + B(G;x,y)(y - x^{2}y - y^{3}),
\end{equation}

\begin{flushleft}
for any graph $G$ of finite order. And is easy to see that the unit circle is always a stable limit cycle for (\ref{system:2}) no matter the choice of $G$. To make it official, let's state it in the following proposition.
\end{flushleft}

\begin{proposition}
The unit circle is always a stable limit cycle for the system (\ref{system:2}) independent of the choice of $G$.
\end{proposition}

\begin{proof}
Assume $G$ is of order $n$. The respective derivative of the Poincare map depends on $\int_{0}^{2\pi}B(G;\cos(t),\sin(t))dt$. Thus we have,

$$\int_{0}^{2\pi}B(G;\cos(t),\sin(t))dt = \sum_{0 \leq i + j \leq n}B_{i,j}(G)\int_{0}^{2\pi}\cos^{i}(t)\sin^{j}(t)dt,$$

where we know that $\int_{0}^{2\pi}\cos^{i}(t)\sin^{j}(t)dt = 0$ if at least one of $i,j$ is odd and $\int_{0}^{2\pi}\cos^{i}(t)\sin^{j}(t)dt > 0$ if both $i,j$ are even. Thus,

$$\int_{0}^{2\pi}B(G;\cos(t),\sin(t))dt = \sum_{0 \leq i + j \leq n}B_{i,j}(G)\int_{0}^{2\pi}\cos^{i}(t)\sin^{j}(t)dt \geq B_{0,0}(G)\int_{0}^{2\pi}dt = 2\pi > 0,$$
which implies the periodic solution at the unit circle is always a stable limit cycle.
\end{proof}

\begin{flushleft}
We now turn our attention to one of the most fundamental bifurcation scenarios concerning limit cycles, the Hopf bifurcation theory for planar analytic systems. In classical operations, this phenomenon is typically analyzed via eigenvalue crossings, centermanifold reductions, and normalform computations. In this work, we reinterpret the Hopf bifurcation conditions through the lens of the boundary polynomial. In particular, by expressing the criticality, transversal-crossing, and amplitude-growth conditions in terms of the boundary polynomial representation, we obtain a streamlined re‐derivation of the standard Hopf conditions while also revealing new structural insights into the parameter dependence of the oscillatory branch. We present the classical Hopf Theorem, reformulate it in the boundary‐polynomial framework, and then demonstrate how this approach naturally leads to our subsequent analysis. In general, a Hopf bifurcation occurs where a periodic orbit is created as the stability of the equilibrium point changes. Consider the system,
\end{flushleft}

\begin{equation}\label{system:3}
\frac{dx}{dt} = \alpha x - y + B(G_{1};x,y) - 1, \frac{dy}{dt} = x + \alpha y + B(G_{2};x,y) - 1,
\end{equation}

\begin{flushleft}
where $G_{1}$, $G_{2}$ are of order $n_{1},n_{2} \geq 2$, respectively, and do not contain isolated vertices. In (\ref{system:3}) for $\alpha = 0$, the Jacobian matrix evaluated at the origin has a pair of purely imaginary eigenvalues, and the origin is a weak focus. Furthermore, $e^{2\pi\alpha}$ is the derivative of the Poincare map at $0$. We now present a very useful result concerning the respective Liapunov number,
\end{flushleft}

\begin{proposition}\label{prop:3}
Consider $G_{1}$,$G_{2}$ to be graphs of order $n_{1},n_{2} \geq 2$, respectively, and do not contain isolated vertices. Then for $\alpha = 0$, the system,
\begin{align*}
\frac{dx}{dt} &= \alpha x - y + B(G_{1};x,y) - 1, \\
\frac{dy}{dt} &= x + \alpha y + B(G_{2};x,y) - 1,
\end{align*}
has Liapunov number $\sigma_{B}$ given by,
\begin{equation}\label{Lyapunov_number}
\begin{split}
\sigma_{B} = \frac{3\pi}{2}\Big(
 & 3B_{0,3}(G_{2}) + B_{1,2}(G_{1}) + B_{2,1}(G_{2}) \\
 & +\, 2B_{0,2}(G_{1})B_{0,2}(G_{2})
   + B_{1,1}(G_{1})B_{0,2}(G_{1})
   - B_{1,1}(G_{2})B_{0,2}(G_{2})
\Big)
\end{split}
\end{equation}
\end{proposition}

This is a very nice boundary-polynomial-dependent reformulation of the result given in Chapter 4.4 of \cite{Perko2001}. Note that $\sigma_{B} \not= 0$ implies the origin is a weak focus of multiplicity one, and if $\sigma_{B} > 0$ or $\sigma_{B} < 0$, it is stable or unstable, respectively, and a Hopf bifurcation occurs at the origin at the bifurcation value $\alpha = 0$. The following is the Hopf bifurcation theorem as stated in \cite{Perko2001},

\begin{theorem}\label{thm:3.1}
If $\sigma_{B} \not= 0$, then a Hopf bifurcation occurs at the origin of the planar analytic system (\ref{system:3}) at the bifurcation value $\alpha = 0$; in particular, if $\sigma_{B} < 0$, then a unique stable limit cycle bifurcates from the origin of (\ref{system:3}) as $\alpha$ increases from zero and if $\sigma_{B} > 0$, then a unique unstable limit cycle bifurcates from the origin of (\ref{system:3}) as $\alpha$ decreases from zero.
\end{theorem}

\begin{remark}\label{remark:3.2}
If $\sigma_{B} < 0$, then it is called a supercritical Hopf bifurcation and if $\sigma_{B} > 0$, then it is called a subcritical Hopf bifurcation.
\end{remark}

The following corollaries present consequences of Proposition (\ref{prop:3}) and Theorem (\ref{thm:3.1}). But before diving into the results, we need to clarify some classical graphs and notation. $S_{n}$ represents the star graph of order $n$. The star graph is a tree on $n$ vertices with one vertex having degree $n - 1$ and the other $n - 1$ vertices having degree $1$. $P_{n}$ denotes the path graph, which is a tree with two vertices of degree $1$, and the other $n - 2$ vertices of degree $2$. $C_{n}$ is the cycle graph, which is a graph on $n$ vertices containing a single cycle through all vertices. $W_{n}$ represents the wheel graph, which is a graph that contains a cycle of order $n - 1$ and for which every graph vertex in the cycle is connected to one other graph vertex. $G \sqcup Q$ denotes the disjoint union between graphs $G$ and $Q$. The notation $kG$ represents the disjoint union of $k$ copies of $G$, i.e., $kG = G \sqcup G \sqcup G \sqcup \cdots \sqcup G$, $k$ times.

\begin{corollary}\label{cor:3.3}
Let $G_{1} = K_{n}$ with $n \geq 3$, and $G_{2} = S_{m} \sqcup P_{2}$, where $m = 2$ or $m \geq 4$. Then there is a unique stable limit cycle that bifurcates from the origin. Furthermore, we can generalize it to $G_{2} = S_{m} \sqcup kP_{2}$ and the result will hold for any $m \geq 2$ as long as we choose $k \geq 2$.
\end{corollary}

\begin{proof}
For $G_{1} = K_{n}$, $n \geq 3$, we have,
$$\sigma_{B} = \frac{3\pi}{2}(3B_{0,3}(G_{2}) + B_{1,2}(G_{1}) + B_{2,1}(G_{2}) - B_{1,1}(G_{2})B_{0,2}(G_{2})).$$
Let's consider $G_{2} = S_{2} \sqcup P_{2}$, for which then we have $B_{0,3}(G_{2}) = B_{2,1}(G_{2}) = 0$, $B_{1,1}(G_{2}) = 4$, $B_{0,2}(G_{2}) = 2$, and thus $\sigma_{B} = \frac{3\pi}{2}(B_{1,2}(G_{1}) - 8)$. We have $B_{1,2}(G_{1}) = 3$ for $n = 3$ and $B_{1,2}(G_{1}) = 0$ for $n > 3$ which shows that $\sigma_{B} < 0$ regardless. Considering $G_{2} = S_{3} \sqcup P_{2}$ we have $B_{0,3}(G_{2}) = B_{2,1}(G_{2}) = B_{0,2}(G_{2}) = 1$, $B_{1,1}(G_{2}) = 4$, and thus $\sigma_{B} = \frac{3\pi}{2}B_{1,2}(G_{1}) \geq 0$. Considering $G_{2} = S_{m} \sqcup P_{2}$ with $m \geq 4$, then we have $B_{0,3}(G_{2}) = B_{2,1}(G_{2}) = 0$, $B_{1,1}(G_{2}) = m+1$, $B_{0,2}(G_{2}) = 1$, thus $\sigma_{B} = \frac{3\pi}{2}(B_{1,2}(G_{1}) - (m+1))$. We have $\sigma_{B} = \frac{3\pi}{2}(2-m) < 0$ for $n = 3$ and $\sigma_{B} = -\frac{3\pi}{2}(m+1) < 0$ for $n > 3$. Therefore, for $G_{2} = S_{m} \sqcup P_{2}$ we have $\sigma_{B} < 0$ when $m = 2$ or $m \geq 4$. In the more general case, we now consider $G_{2} = S_{m} \sqcup k P_{2}$. For $m = 2$, then $B_{0,3}(G_{2}) = B_{2,1}(G_{2}) = 0$, $B_{1,1}(G_{2}) = 2 + 2k$, $B_{0,2}(G_{2}) = k + 1$, thus $\sigma_{B} = \frac{3\pi}{2}(B_{1,2}(G_{1}) - (2+2k)(k + 1))$ where $\sigma_{B} = -\frac{3\pi}{2}(2k^{2} + 4k - 1) < 0$ for $n = 3$ and $\sigma_{B} = -\frac{3\pi}{2}(2+2k)(k+1) < 0$ for $n > 3$. For $m = 3$ then, $B_{0,3}(G_{2}) = B_{2,1}(G_{2})= 1$, $B_{1,1}(G_{2}) = 2 + 2k$, $B_{0,2}(G_{2}) = k$ and thus $\sigma_{B} = -\frac{3\pi}{2}(2k^{2} + 2k - B_{1,2}(G_{1}) - 4)$. If $n = 3$, we have $\sigma_{B} = -\frac{3\pi}{2}(2k^{2} + 2k -7) < 0$ when $k \geq 2$, and if $n > 3$, we have $\sigma_{B} = -3\pi(k^{2} + k - 2) < 0$ when $k \geq 2$. For $m \geq 4$, then $B_{0,3}(G_{2}) = B_{2,1}(G_{2}) = 0$, $B_{1,1}(G_{2}) = m - 1 + 2k$, $B_{0,2}(G_{2}) = k$ and thus $\sigma_{B} = \frac{3\pi}{2}(B_{1,2}(G_{1}) - (m-1+2k)k)$. If $n = 3$ then $\sigma_{B} = \frac{3\pi}{2}(3 - (m - 1 + 2k)k) < 0$ and if $n > 3$ then $\sigma_{B} = -\frac{3\pi}{2}(m-1+2k)k < 0$. Therefore, for $G_{2} = S_{m} \sqcup kP_{2}$ with $m \geq 2$, we can guarantee that $\sigma_{B} < 0$ if we always choose $k \geq 2$.
\end{proof}

\begin{remark}\label{remark:3.4}
In the generalized $G_{2}$ of (\ref{cor:3.3}), choosing $k \geq 2$ is to ensure that regardless of the order $m \geq 2$, the Liapunov number $\sigma_{B}$ is always negative. But this choice was necessary because of when $m = 3$. But if we only choose $m = 2$ or $m \geq 4$ then no matter the choice of $k$, the Liapunov number is negative.
\end{remark}

The following corollary is simple, but very important to state.

\begin{corollary}\label{cor:3.5}
If $G_{2}$ is the family of complete graphs of order at least $3$, then $\sigma_{B} \not< 0$.
\end{corollary}

\begin{proof}
Having $G_{2} = K_{n}$, with $n \geq 3$, implies that $B_{1,1}(G_{2}) = 0$ by convention.
\end{proof}

\begin{remark}\label{remark:3.6}
Another more general way to interpret corollary (\ref{cor:3.5}) is that it cannot be guaranteed for a unique stable limit cycle to bifurcate from the origin if $G_{2}$ has no vertices with degree $1$. We can officially state it in the following:
\end{remark}

\begin{corollary}\label{cor:3.7}
If $G_{2}$ has no pendant vertices, then $\sigma_{B} \not< 0$. Furthermore, the same result holds if there are no components of $G_{2}$ isomorphic to $P_{2}$.
\end{corollary}

\begin{proof}
$B_{1,1}(G_{2})$ represents the number of pendant vertices, i.e., vertices with degree $1$. If $G_{2}$ has no pendant vertices, then $B_{1,1}(G_{2}) = 0$ and thus $\sigma_{B} \not< 0$. Now, if $G_{2}$ has no component isomorphic to $P_{2}$, it is clear by \cite{Carballosa2025} that $B_{0,2}(G_{2}) = 0$, which eliminates the only negative term in $\sigma_{B}$.
\end{proof}

\begin{corollary}\label{cor:3.8}
Consider $G_{2} = P_{2}$.
\begin{enumerate}[label=\textbf{\alph*)}]
    \item If $G_{1}$ is either $K_{3}$, $C_{3}$, or $S_{n}$ with $n \geq 2$, then $\sigma_{B} > 0$.

    \item If $G_{1}$ is either $K_{n}$, $C_{n}$, or $W_{n}$, with $n \geq 4$, then $\sigma_{B} < 0$.
\end{enumerate}
\end{corollary}

\begin{proof}
Consider $G_{2} = P_{2}$, for which then,

$$\sigma_{B} = \frac{3\pi}{2}(B_{1,2}(G_{1}) + B_{0,2}(G_{1})(2 + B_{1,1}(G_{1})) - 2).$$

We have that $K_{3} = C_{3}$ and thus let $G_{1} = K_{3} = C_{3}$, which implies $B_{1,2}(G_{1}) = 3$, $B_{0,2}(G_{1}) = 0$, and thus $\sigma_{B} = \frac{3\pi}{2} > 0$. Let $G_{1} = S_{n}$ with $n \geq 2$, then if $n = 2$ we have $B_{1,2}(G_{1}) = 0$, $B_{0,2}(G_{1}) = 1$, $B_{1,1}(G_{1}) = 2$, and thus $\sigma_{B} = 3\pi > 0$. If $n = 3$, then $B_{1,2}(G_{1}) = 3$, $B_{0,2}(G_{1}) = 0$, and hence $\sigma_{B} = \frac{3\pi}{2} > 0$. If $n \geq 4$, then $B_{1,2}(G_{1}) = \binom{n-1}{2} = \frac{(n-1)(n-2)}{2}$, $B_{0,2}(G_{1}) = 0$, and thus $\sigma_{B} = \frac{3\pi}{4}((n-1)(n-2)-4) > 0$. Therefore, $G_{1} = S_{n}$ with $n \geq 2$ guarantees $\sigma_{B} > 0$. Now, choose $G_{1}$ to be either $K_{n}$, $C_{n}$, or $W_{n}$, with $n \geq 4$, then $B_{1,2}(G_{1}) = B_{0,2}(G_{1}) = 0$, and thus $\sigma_{B} = -3\pi < 0$.
\end{proof}

The following corollary shows when one graph is chosen for Proposition \ref{prop:3} and how the value of $\sigma_{B}$ changes across the classical graphs.

\begin{corollary}\label{cor:3.9}
Consider $G = G_{1} = G_{2}$.
\begin{enumerate}[label=\textbf{\alph*)}]
    \item For $G = K_{n}$ or $G= W_{n}$, with $n \geq 4$, we have $\sigma_{B} = 0$.
    \item For $G = C_{n}$ with $n \geq 4$, we have $\sigma_{B} = \frac{3n\pi}{2}$.
    \item For $G = S_{2}$ we have $\sigma_{B} = 3\pi$ and for $G = S_{3}$ we have $\sigma_{B} = \frac{21\pi}{2}$. Furthermore, for $G = S_{n}$ with $n \geq 4$ we have $\sigma_{B} = \frac{3(n-1)(n-2)\pi}{4}$.
\end{enumerate}
\end{corollary}

\begin{proof}
Let $G = G_{1} = G_{2}$. Then we have,

$$\sigma_{B} = \frac{3\pi}{2}(3B_{0,3}(G) + B_{1,2}(G) + B_{2,1}(G) + 2(B_{0,2}(G))^{2}) \not< 0.$$

If we choose $G = K_{n}$ or $G = W_{n}$, with $n \geq 4$, is clear that $B_{0,3}(G) = B_{1,2}(G) = B_{2,1}(G) = B_{0,2}(G) = 0$ due to the structure of the graphs, and hence $\sigma_{B} = 0$. If on the other hand we have $G = C_{n}$ with $n \geq 4$, then $B_{0,3}(G) = B_{1,2}(G) = B_{0,2}(G) = 0$ and $B_{2,1}(G) = n$ which simplifies to $\sigma_{B} = \frac{3n\pi}{2}$. If $G = S_{2}$ then $B_{0,3}(G) = B_{1,2}(G) = B_{2,1}(G) = 0$, $B_{0,2}(G) = 1$, which gives $\sigma_{B} = 3\pi$. If $G = S_{3}$ then $B_{0,3}(G) = 1$, $B_{1,2}(G) = 3$, $B_{2,1}(G) = 1$, $B_{0,2}(G) = 0$, which gives $\sigma_{B} = \frac{21\pi}{2}$. For $G_{1} = S_{n}$ with $n \geq 4$ we have $B_{0,3}(G) = B_{2,1}(G) = B_{0,2}(G) = 0$, $B_{1,2}(G) = \binom{n-1}{2} = \frac{(n-1)(n-2)}{2}$, and hence $\sigma_{B} = \frac{3(n-1)(n-2)\pi}{4}$.
\end{proof}

\begin{corollary}\label{cor:3.10}
If $G_{1}$ has order $n_{1} \geq 3$ with no pendant vertices, is connected, no pair of vertices has exactly one-single other connection to a vertex, and if $G_{2}$ has order $n_{2} \geq 3$, is connected, and has at least one vertex of degree $2$, then we have a subcritical Hopf bifurcation. Furthermore, if $G_{2}$ has at least one component isomorphic to $P_{2}$, no components isomorphic to $P_{3}$ or $C_{3}$, has no vertex of degree $2$, and has at least one pendant vertex, then we have a supercritical Hopf bifurcation.
\end{corollary}

\begin{proof}
    For this proof we are using Theorem 2.2 of (\cite{Carballosa2025}). Let $G_{1}$ be of order $n_{1} \geq 3$ with no pendant vertices, that is, $B_{1,1}(G_{1}) = 0$. Adding the condition that $G_{1}$ be connected, then we have $B_{0,j}(G_{1}) = 0$ for $1 \leq j \leq n_{1}-1$, which always includes $B_{0,2}(G_{1}) = 0$. Now, $G_{1}$ having no pair of vertices that has exactly one-single other connection to a vertex implies that $B_{1,2}(G_{1}) = 0$. Thus,
    $$\sigma_{B} = \frac{3\pi}{2}(3B_{0,3}(G_{2}) + B_{2,1}(G_{2}) - B_{1,1}(G_{2})B_{0,2}(G_{2})).$$

    For the first case, let $G_{2}$ be of order $n_{2} \geq 3$, be connected, and has at least one vertex of degree $2$, then $B_{0,2}(G_{2}) = 0$ and $B_{2,1}(G_{2}) > 0$, which is enough to guarantee that $\sigma_{B} > 0$.

    For the second case, let $G_{2}$ be a graph with at least one component isomorphic to $P_{2}$, which implies $p := B_{0,2}(G) > 0$. Adding the restriction that $G_{2}$ has no components isomorphic to $P_{3}$ or $C_{3}$, directly implies $q := B_{0,3}(G_{2}) = 0$. Now, $G_{2}$ having no vertex of degree $2$, and it has at least one pendant vertex implies $B_{2,1}(G_{2}) = 0$ and $B_{1,1}(G_{2}) > 0$. Thus, giving $\sigma_{B} = -\frac{3\pi}{2}B_{1,1}(G_{2})B_{0,2}(G_{2}) < 0$.
\end{proof}

\begin{remark}\label{remark:3.11}
    To give some reassurance that the graphs as claimed in Corollary (\ref{cor:3.10}) exists, note that $C_{4}$ satisfies the condition for $G_{1}$, and $2P_{2}$ satisfies the condition for the second version of $G_{2}$.
\end{remark}

\begin{figure}[htbp]
  \centering

  \begin{minipage}[t]{0.32\textwidth}
    \centering
    \begin{tikzpicture}[
      every node/.style={circle, draw, semithick, minimum size=4mm, inner sep=1pt, font=\scriptsize}
    ]
      % K4
      \node (v1) at (-0.9,  0.9) {$v_1$};
      \node (v2) at ( 0.9,  0.9) {$v_2$};
      \node (v3) at (-0.9, -0.9) {$v_3$};
      \node (v4) at ( 0.9, -0.9) {$v_4$};

      \draw[semithick] (v1)--(v2);
      \draw[semithick] (v1)--(v3);
      \draw[semithick] (v1)--(v4);
      \draw[semithick] (v2)--(v3);
      \draw[semithick] (v2)--(v4);
      \draw[semithick] (v3)--(v4);
    \end{tikzpicture}
  \end{minipage}\hfill
  \begin{minipage}[t]{0.32\textwidth}
    \centering
    \begin{tikzpicture}[
      every node/.style={circle, draw, semithick, minimum size=4mm, inner sep=1pt, font=\scriptsize}
    ]
      %S2 (two vertices, one edge)
      \node (u1) at (-1.4,  0.0) {$u_1$};
      \node (u2) at (-0.6,  0.0) {$u_2$};
      \draw[semithick] (u1)--(u2);

      %First P2
      \node (u3) at ( 0.4,  0.6) {$u_3$};
      \node (u4) at ( 1.2,  0.6) {$u_4$};
      \draw[semithick] (u3)--(u4);

      %Second P2
      \node (u5) at ( 0.4, -0.6) {$u_5$};
      \node (u6) at ( 1.2, -0.6) {$u_6$};
      \draw[semithick] (u5)--(u6);
    \end{tikzpicture}
  \end{minipage}\hfill
  \begin{minipage}[t]{0.32\textwidth}
    \centering
    \includegraphics[width=\linewidth]{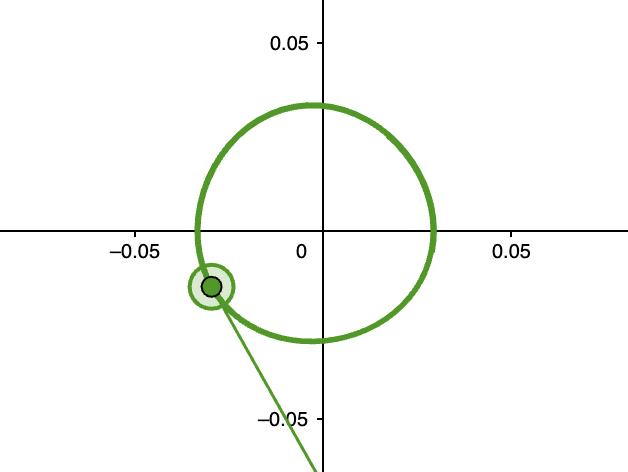}
  \end{minipage}

  \caption{Example with $G_{1}=K_{4}$, $G_{2}=S_{2}\sqcup 2P_{2}$, and $\alpha = 0.001$ for (\ref{system:3}).}
  \label{fig:K4_vs_S2_2P2}
\end{figure}
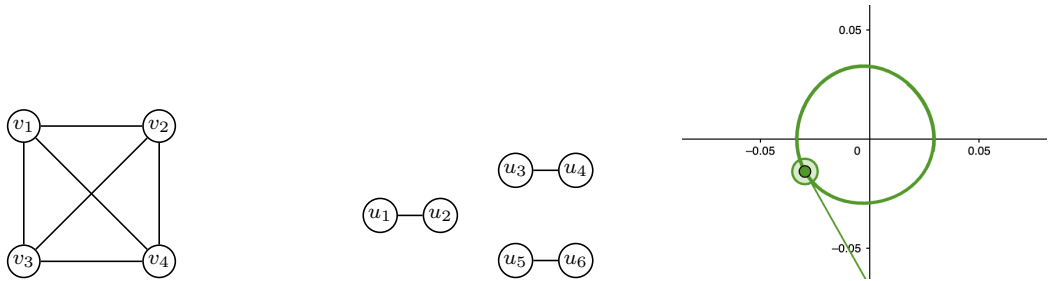

Concerning Proposition \ref{prop:3}, if $\sigma_{B} = \alpha = 0$, then the origin will be a weak focus of multiplicity $g > 1$. It is proven in \cite{Andronov} that at most $g$ limit cycles can bifurcate from the origin as $\alpha$ varies through the bifurcation value $\alpha = 0$. Furthermore, there is an analytic perturbation of,

\begin{equation}\label{equation:14}
    \begin{aligned}
\frac{dx}{dt} &= - y + B(G_{1};x,y) - 1, \\
\frac{dy}{dt} &= x + B(G_{2};x,y) - 1,
\end{aligned}
\end{equation}

which causes exactly $g$ limit cycles to bifurcate from the origin at $\alpha = 0$. We are maintaining $G_{1},G_{2}$ to have the same properties as in Proposition \ref{prop:3}. The next theorem is explicitly stated in \cite{Perko2001}.

\begin{theorem}\label{thm:3.12}
If the origin is a multiple focus of multiplicity $g$ of the system (\ref{equation:14}), then for $k \geq 2g + 1$ there is a $\delta > 0$ and $\epsilon > 0$ such that any system $\epsilon$-close to (\ref{equation:14}) in the $C^{k}$-norm has at most $g$ limit cycles in $N_{\delta}(0)$ and for any $\delta > 0$ and $\epsilon > 0$ there is an analytic system which is $\epsilon$-close to (\ref{equation:14}) in the $C^{k}$-norm and has exactly $g$ simple limit cycles in $N_{\delta}(0)$.
\end{theorem}

\begin{remark}\label{remark:3.13}
    $N_{\delta}(0)$ is the $\delta$-neighborhood of the origin and for a vector field $f \in C^{k}(\Omega)$ where $\Omega$ is an open subset of $\mathbb{R}^{n}$, the $C^{k}$-norm is defined as,

    \begin{equation}
        \|f\|_{k} = \sup_{\Omega}|f| + \sup_{\Omega}\|Df\| + \cdots + \sup_{\Omega}\|D^{k}f\|,
    \end{equation}
    where $\|D^{k}f\| = \max|\frac{\partial^{k}f}{\partial x_{j_{1}}\cdots \partial x_{j_{k}}}|$. The maximum is taken over $j_{1},\ldots,j_{k}$ = $1,\ldots,n$. Note that each of the spaces of functions in $C^{k}(\Omega)$, bounded in the $C^{k}$-norm, is a Banach space and $C^{k+1}(\Omega) \subset C^{k}(\Omega)$ for $k = 0,1,\ldots$.
\end{remark}

The following Proposition shows a generalization of the proposition \ref{prop:3}, which has its roots in Chapter 4.4 of \cite{Perko2001}.

\begin{proposition}\label{prop:4}
    Consider $G_{1}$,$G_{2}$ to be graphs of order $n_{1},n_{2} \geq 2$, respectively, and do not contain isolated vertices. Let $\sigma = (\alpha_{1}\alpha_{4}-\alpha_{2}\alpha_{3})$, $\alpha_{1} + \alpha_{4} = 0$, such that we have the following planar analytic system,
    
    \begin{align*}
    \frac{dx}{dt} &= \alpha_{1}x + \alpha_{2}y + B(G_{1};x,y) - 1, \\
    \frac{dy}{dt} &= \alpha_{3}x + \alpha_{4}y + B(G_{2};x,y) - 1,
\end{align*}
where the Jacobian matrix evaluated at the origin will have a pair of imaginary eigenvalues and the origin will be a weak focus. The generalized Liapunov number $\widetilde{\sigma}_{B}$ is given by
\begin{equation}\label{generalized_Lyapunov_number}
\begin{split}
\widetilde{\sigma}_{B} = -\frac{3\pi}{2\alpha_{2}\sigma^{\frac{3}{2}}}\Big[
\big[\alpha_{1}\alpha_{3}((B_{1,1}(G_{1}))^{2} + B_{1,1}(G_{1})B_{0,2}(G_{2}) + B_{0,2}(G_{1})B_{1,1}(G_{2})) \\
+\alpha_{1}\alpha_{2}((B_{1,1}(G_{2}))^{2}+B_{1,1}(G_{1})B_{0,2}(G_{2})) + \alpha_{3}^{2}(B_{1,1}(G_{1})B_{0,2}(G_{1})+2B_{0,2}(G_{1})B_{0,2}(G_{2})) \\
-2\alpha_{1}\alpha_{3}(B_{0,2}(G_{2}))^{2} + (\alpha_{2}\alpha_{3}-2\alpha_{1}^{2})B_{1,1}(G_{2})B_{0,2}(G_{2})\big] \\
-(\alpha_{1}^{2}+\alpha_{2}\alpha_{3})\big[3\alpha_{3}B_{0,3}(G_{2}) + 2\alpha_{1}(B_{2,1}(G_{1}) + B_{1,2}(G_{2})) + (\alpha_{3}B_{1,2}(G_{1}) - \alpha_{2}B_{2,1}(G_{2}))\big]\Big].
\end{split}
\end{equation}
\end{proposition}

Concerning Proposition \ref{prop:4}, if we have $\widetilde{\sigma}_{B} \not= 0$, then Theorem \ref{thm:3.1} still holds with $\alpha = \alpha_{1} + \alpha_{4}$.

\end{document}